\documentclass[11pt]{article}
\usepackage[OT2,OT1]{fontenc}
\usepackage{amsfonts,amsmath, amssymb,latexsym,mathrsfs}
\usepackage[all,cmtip]{xy}
\usepackage{float}
\def\Z{{\mathbb Z}}

\def\GL{{\rm GL}}

\def\Jac{{\rm Jac}}

\def\Cl{{\rm Cl}}

\def\OO{{\mathcal O}}
\def\OO{{\mathcal O}}

\def\P{{\mathbb P}}
\def\Disc{{\rm Disc}}

\def\rank{{\rm rank}}
\def\Sel{{\rm Sel}}

\def\Vol{{\rm Vol}}
\def\R{{\mathbb R}}
\def\F{{\mathbb F}}

\def\Q{{\mathbb Q}}

\def\Z{{\mathbb Z}}
\def\P{{\mathbb P}}
\def\F{{\mathbb F}}
\def\Q{{\mathbb Q}}
\def\C{{\mathbb C}}

\def\wzn2{{W_{\Z,+}^{(2-)}}}

\def\fz1{{F_{\Z,1}}}
\def\pic{{\rm Pic}}

\def\div{{\rm div}}
\def\res{{\rm Res}}
\def\rk{{\rm rk}}

\newtheorem{theorem}{Theorem}[section]

\newtheorem{lemma}[theorem]{Lemma}
\newtheorem{remark}[theorem]{Remark}

\newenvironment{proof}{\noindent {\bf Proof:}}{$\Box$ \vspace{2 ex}}

\begin{document}

\title{Bounds on $2$-torsion in class groups of number fields \\ and integral points on elliptic curves} 

\author{M.\ Bhargava, A.\  Shankar, T.\ Taniguchi, F.\ Thorne, J.\ Tsimerman, and Y.\ Zhao}

\maketitle
\begin{abstract}
  We prove the first known nontrivial bounds on the sizes of the 2-torsion
  subgroups of the class groups of cubic~and~higher degree number
  fields $K$ (the trivial bound being $O_{\epsilon}(|\Disc(K)|^{1/2+\epsilon})$ by
  Brauer--Siegel).  This~yields corresponding improvements to: 1) bounds of Brumer and Kramer on
  the sizes of 2-Selmer groups and ranks of elliptic curves; 2) bounds of Helfgott and Venkatesh on
  the number of integral points on elliptic curves; 3) bounds on the sizes of 2-Selmer groups and ranks of Jacobians of hyperelliptic curves; and 4) bounds of Baily and
  Wong on the number of $A_4$-quartic fields of bounded
  discriminant.
\end{abstract}

\setcounter{tocdepth}{4}


\section{Introduction}

For any number field $K$ of a given degree $n$, and any positive integer $m$, it is
conjectured that the size $h_m(K)$ of the $m$-torsion subgroup of the
class group of $K$ satisfies $h_m(K)=O_\epsilon(|\Disc(K)|^\epsilon)$.  It~clearly suffices to
prove the conjecture in the case that $m=p$ is a prime.  This is
currently known only for $n=2$ and $p=2$ via Gauss's genus theory.
For other degrees and primes $p$, the best that is known in general is
that $h_p(K)=O_\epsilon(|\Disc(K)|^{1/2+\epsilon})$, by Brauer--Siegel.  It has
thus become a problem of interest to break the trivial bound for any
$n$ and $p$ for $(n,p)\neq (2,2)$. 

The first breakthrough on this problem was obtained in the case of 
$(n,p)=(2,3)$, achieved by Pierce~\cite{Pierce} and Helfgott and
Venkatesh~\cite{HV}, who proved bounds of $O_\epsilon(|\Disc(K)|^{27/56+\epsilon})$ and $O(|\Disc(K)|^{0.44178\ldots})$, respectively. 
Ellenberg and Venkatesh~\cite{EV} improved these bounds to \linebreak $O_\epsilon(|\Disc(K)|^{1/3+\epsilon})$, and also achieved power savings
over the trivial bounds for $(n,p)=(3,3)$ (again~proving $O_\epsilon(|\Disc(K)|^{1/3+\epsilon})$) and also for $(n,p)=(4,3)$.  These are the only three cases of $(n,p)$ previously known.

The primary purpose of this paper is to 
develop a uniform method that
yields nontrivial bounds for $(n,2)$ for {\it all} degrees $n>2$, i.e., nontrivial
bounds on the $2$-torsion in class groups of number fields of general
degree $n$. More precisely, we prove the following theorem. 

\begin{theorem}\label{main}
  The size $h_2(K)$ of the $2$-torsion subgroup of the class group of
  a number field $K$ of degree $n$ is $O_{\epsilon}(|\Disc(K)|^{1/2-\delta_n+\epsilon})$ for
  some $\delta_n>0$.
\end{theorem}
For cubic and quartic fields $K$, our methods in fact allow us to prove that
$h_2(K)=\linebreak O_\epsilon(|\Disc(K)|^{.2784\ldots+\epsilon})$, which is a significant improvement over
$O_\epsilon(|\Disc(K)|^{1/2+\epsilon})$.   For $n>4$, we prove Theorem~\ref{main} with
$\delta_n=\frac{1}{2n}$.

As a consequence of Theorem~\ref{main} for $n=3$, we obtain
corresponding improvements to the best known bounds on the 2-Selmer
groups, ranks, and the number of integral points on elliptic curves,
in terms of their discriminants.
Specifically, we prove:

\begin{theorem} \label{main2}
  Let $E$ be an elliptic curve over $\Q$ in Weierstrass form with integral coefficients.
Then:
\begin{itemize}
\item[{\rm (a)}]
$|\Sel_2(E)|=O_\epsilon(|\Disc(E)|^{.2784\ldots+\epsilon})$;
\item[{\rm (b)}]
$\rank(E)< c_\epsilon+(.2784\ldots+\epsilon)\log_2(|\Disc(E)|)$ for some constant $c_\epsilon$ depending only on $\epsilon$;
\item[{\rm (c)}]
 The number of integral points on $E$ is at most
$
O_{\epsilon}(|\Disc(E)|^{0.1117\ldots+\epsilon}).$
\end{itemize}
\end{theorem}
This significantly improves on the previously best known bounds of
$O_\epsilon(|\Disc(E)|^{1/2+\epsilon})$, $c_\epsilon+(1/2+\epsilon)\log_2(|\Disc(E)|)$, and
$O_\epsilon(|\Disc(E)|^{.2007\ldots+\epsilon})$, respectively, as obtained in the works of Brumer and
Kramer~\cite{BK} and Helfgott and Venkatesh~\cite{HV}.  Our methods 
follow theirs, but we employ our new bounds on the 2-torsion in the class
groups of cubic fields to obtain the improvements.


We obtain analogous improvements on the best known bounds on the sizes of 2-Selmer groups and ranks of Jacobians of general hyperelliptic curves over $\Q$ of given genus.  
 It is known (generalizing the aforementioned work of Brumer and Kramer) that if $C:y^2=f(x)$ is a hyperelliptic curve over $\Q$, where $f$ is separable of degree~$n$ over~$\Q$, and $K$ is the \'etale $\Q$-algebra $\Q[x]/(f(x))$, then $|\Sel_2(\Jac(C))|=O_\epsilon(|\Disc(K)|^\epsilon h_2(K))$, and thus $|\Sel_2(\Jac(C))|=O_\epsilon(|\Disc(K)|^{1/2+\epsilon})$ by Brauer--Siegel.  
 This yields the best known bounds unconditionally on the 2-Selmer groups and ranks of Jacobians of hyperelliptic curves over $\Q$ to date.  Theorem~\ref{main} breaks these ``trivial'' bounds:

 \begin{theorem} \label{main2.5}
  Let $C:y^2=f(x)$ be a hyperelliptic curve over $\Q$, where $f(x)\in\Z[x]$ is a separable polynomial of degree $n$, and let $K=\Q[x]/(f(x))$ Then:
\begin{itemize}
\item[{\rm (a)}]
$|\Sel_2(\Jac(C))|=O_\epsilon(|\Disc(K)|^{1/2-\delta_n+\epsilon});$
 \item[{\rm (b)}]
$\rank(\Jac(C))<c_\epsilon+(1/2-\delta_n+\epsilon)\log_2|\Disc(K)|$ for some constant $c_\epsilon$ depending only on $\epsilon$.
\end{itemize}
\end{theorem}

  Since $|\Disc(K)|< |\Disc(C)|$, where $\Disc(C)$ is the discriminant of the curve $C$, we may also replace $|\Disc(K)|$ by $|\Disc(C)|$ in each of the bounds in Theorem~\ref{main2.5}.  We suspect that the analogues of Theorem~\ref{main2}(c) for hyperelliptic curves could also be developed using Theorem~\ref{main}, though we do not work this out here (but see forthcoming work of Alpoge).

As a further consequence of Theorem~\ref{main} for $n=3$, we obtain
the best known bounds on the number of $A_4$-quartic fields of bounded
discriminant:

\begin{theorem}\label{main3}
Let $N_4(A_4,X)$ denote the number of isomorphism classes of quartic fields having associated Galois
group $A_4$ and discriminant less than $X$.  Then $N_4(A_4,X)=O_\epsilon(X^{.7784\ldots+\epsilon})$.
\end{theorem}
This improves on the previously best known bounds of
$O_\epsilon(X^{1+\epsilon})$ and $O_\epsilon(X^{5/6+\epsilon})$ due to Baily~\cite{Baily} and
Wong~\cite{Wong}, respectively.  We also obtain the same bound for the number
of quartic fields of bounded discriminant having any given fixed quadratic
resolvent field (cf.\ Remark~\ref{quadresolve}). 

Our method for bounding the 2-torsion in class groups of number
fields $K$ is relatively elementary, and involves two key results.  The
first result states that we may always suitably choose an ideal $I$ in
any 2-torsion ideal class of a number field $K$ so that its square
$I^2$ has a generator $\beta$ whose norm is small (specifically,
$\ll |\Disc(K)|$) and whose archimedean embeddings are all of
comparable size (specifically, $\ll |\Disc(K)|^{1/n}$):

\begin{theorem}\label{firstres}
Let $K$ be a number field of degree $n$, and let $M_\infty(K)$ denote the set of
infinite places of~$K$.  Then any ideal class of order $2$ in the ring of integers $\OO_K$ of $K$ has a representative
ideal $I\subset \OO_K$ such that $I^2=\beta\OO_K$ and $|\beta|_v\leq |\Disc(K)|^{1/n}$ for all 
$v\in M_\infty(K)$. 
\end{theorem}

The second result states that the ring of integers $\OO_K$ of a number field $K$ of degree $n$, with $r$ real and $s$ pairs of complex embeddings, cannot be too skew when viewed as a lattice in $\R^n=\R^r\times \C^s$ via its archimedean embeddings up to conjugacy, where we identify $\C$ with the real vector space $\R^2$ in the usual way. For a general integral lattice, the largest successive minimum is bounded by its covolume. However, we prove that one can do quite a bit better than the covolume $2^{-s}|\Disc(K)|^{1/2}$ when the lattice in question is the ring of integers $\OO_K$ in a number field $K$. Specifically, in Section~3 we prove the following theorem.

\begin{theorem}\label{secondres}
Let $K$ be a number field of degree $n$. The length of the largest successive minimum of the lattice $\OO_K\subset \R^n$ is $O(|\Disc(K)|^{1/n})$.  
\end{theorem}
The exponent in Theorem~\ref{secondres} is optimal, as can be seen, for example, via the fields $\Q(\sqrt[n]{p})$. 

These two results together immediately imply, by elementary geometry-of-numbers arguments, 
an upper estimate of the
form $O(|\Disc(K)|^{1/2})$ for the number of possible generators
$\beta$ of~$I^2$, which is already an improvement upon the trivial
estimate from Brauer--Siegel. We then obtain a final saving by providing a power-saving upper estimate for the count of just those elements in the region $\{\beta\in\OO_K:|\beta_v|\leq |\Disc(K)|^{1/n}\:\forall v\in M_\infty(K)\}$ that have square norm, for which we appeal to a result of Bombieri and Pila~\cite{BP} on the number of integral points in bounded regions on curves. In the case $n=3$, we can do even better by appealing to the work of Helfgott and Venkatesh~\cite{HV} on integral points on elliptic curves, and via the arguments of~\cite{T} these improvements are carried over to the case $n=4$.

An analogous proof can in fact be carried out over function fields, with all the difficulties in the geometry of numbers replaced by the Riemann--Roch theorem. As a result, we obtain a nontrivial upper bound on the 2-torsion in the class groups of function fields in which the dependence is purely upon the genus. Specifically, we prove:

\begin{theorem}\label{main5}
Let $C$ be a curve of genus $g$ over a finite field $k$. Then $$|\pic^0(C)(k)[2]|\leq \frac{|k|^{g+1}-1}{|k|-1}.$$
Moreover, if $C$ admits a degree $n$ map to $\P^1$ over $k$, then $$|\pic^0(C)(k)[2]|\ll_n |k|^{(1-\frac{1}{n})g}.$$
\end{theorem}
The second part of the theorem is obtained by appealing to a function-field analogue of the aforementioned result of Bombieri and Pila~\cite{BP} established by Sedunova~\cite{Sedunova}. 


This paper is organized as follows. In Section~2, we discuss convenient choices of ideal representatives for ideal classes of order $m$ in number fields of degree $n$, and we prove Theorem~\ref{firstres}.  In Section~3, we then consider the lattice structure and successive minima in rings of integers in number fields, and prove Theorem~\ref{secondres}.  We combine these two key results in Section~4, together with the method of Bombieri and Pila, to prove our main Theorem~\ref{main}. In Section~5, we then appeal to the work of Helfgott and Venkatesh and the fifth-named author to obtain the stated improvements to Theorem~\ref{main}
in the cubic and quartic cases, respectively, and we also prove Theorem~\ref{main2}. In Section~6, we apply our results to improve the best known bounds on the number
of $A_4$-quartic fields of bounded discriminant. Finally, in Section 7, we discuss the function field analogue and prove Theorem~\ref{main5}.

\section{Representative ideals $I$ for $m$-torsion ideal classes such that $I^m$ has a generator $\beta$ of small norm 
and archimedean embeddings of comparable size}
\label{2torsec}

In this section, we prove that in any $m$-torsion
ideal class of a number field $K$, there is an ideal $I$ such that $I^m$ 
has a generator $\beta$ which is both of small norm and whose
archimedean embeddings are all of comparable size.  
More precisely, we prove the following theorem:

\begin{theorem}\label{firstres2}
Let $K$ be a number field of degree $n$, and let $M_\infty(K)$ denote the set of infinite places of $K$. 
Then any ideal class of order $m$ in the ring of integers $\OO_K$ of $K$ has a representative
ideal $I\subset \OO_K$ such that $I^m=\beta\OO_K$ and $|\beta|_v\:\leq\:|\Disc(K)|^{m/2n}$ for all $v\in M_\infty(K)$.
\end{theorem}

\begin{proof}
  Suppose that $[I]$ is an ideal class of $\OO_K$ of order
  $m$. By assumption, we may write $I^m=(\alpha)$ for some element
  $\alpha\in\OO_K$. However, we would like to find an element
  $\beta\in\OO_K$ such that $\beta=\alpha\kappa^m$ for some $\kappa\in
  K^\times$ (so that $\beta$ is
  a generator of $(\kappa I)^m$ where $\kappa I$ is in the same ideal
  class as $I$), and such that $\beta$ is both of small norm and has 
  archimedean embeddings that are all of comparable size. 

To prove 
that such an element $\beta=\alpha\kappa^m$ exists, 
choose a nonzero element $\kappa\in I^{-1}$ such that $$|\kappa|_v\leq
|\Disc(K)|^{1/2n}|\alpha|_v^{-1/m}$$ for all $v\in M_\infty(K)$; such a $\kappa\in I^{-1}$ exists
  by Minkowski's Theorem, since the lattice $I^{-1}$,  when embedded again in $\R^n\cong R^r\times \C^s$ via the $r$ real and $s$ complex embeddings of $K$ up to conjugacy, has covolume $2^{-s}|\Disc(K)|^{1/2}|N(I)|^{-1}$, 
  and
  we are searching for $\kappa$ in a centrally symmetric convex body of volume 
  $$\prod_{v\rm{\;real}}(2\,|\Disc(K)|^{1/2n}|\alpha|_v^{-1/m})
  \prod_{v\rm{\;complex}}(\sqrt{\pi}\,|\Disc(K)|^{1/2n}|\alpha|_v^{-1/m})
  \geq
  2^{n}\cdot2^{-s}\,|\Disc(K)|^{1/2}N(I)^{-1}.$$ 

For such an element $\kappa$, we have 
$|\alpha\kappa^m|_v\leq 
|\Disc(K)|^{m/2n}$. The fact that $\kappa\in I^{-1}$ ensures that $\beta=\alpha\kappa^m$
is in $\OO_K$. 
Thus, we have produced a
  $\beta\in\OO_K$ of the form $\alpha\kappa^m$ whose archimedean embeddings all have size bounded by
  $|\Disc(K)|^{m/2n}$, as desired.
\end{proof}

\noindent
The case $m=2$ of Theorem~\ref{firstres2} yields Theorem~\ref{firstres} as stated in the introduction.

\section{Regularity of shapes of rings of integers of number fields}

In order to use Theorem~\ref{firstres} to bound the number of $2$-torsion classes in class groups of number fields, it will be important for us to know that the lattices underlying the ring of integers in number fields cannot be too ``skew''.  We now make this precise. 

Let $K$ be a number field of a given degree $n$ with $r$ real and $s$ conjugate pairs of complex embeddings, and consider again the  embedding of $K$ into $\R^n=\R^r\times \C^s$ via its archimedean embeddings up to complex conjugation, where as before we identify $\C$ with~$\R^2$ in the usual way.
Then the ring of integers $\OO_K$ forms a lattice in $\R^n$ of covolume $2^{-s}|\Disc(K)|^{1/2}$. The length $|\cdot|$ induced by the $($conjugate$)$ trace pairing equals
the sum of the squares of the norms of all its archimedean
absolute values, i.e., the usual length induced by the embedding of $\OO_K$ into~$\R^n$.
  
It is also known (see, e.g., \cite[Lecture X, \S 6]{Siegel}) that there exists a {\itshape Minkowski reduced basis} $\{1,v_1,v_2,\ldots,v_{n-1}\}$ for $\OO_K$
with $1 \leq |v_1| \leq \cdots \leq |v_{n - 1}|$ and such that $|v_i| \ll |v|$ for every $i \leq {n - 2}$ and every
$v \not \in \textnormal{Span}\{1, v_1, \dots, v_{i - 1}\}$, and which further satisfies
\[
|v_1| \cdots |v_{n - 1}| \asymp |\Disc(K)|^{1/2}.
\]

We now prove that the lattice $\OO_K$ is not too skew -- more specifically, we prove that
the largest successive minimum $v_{n - 1}$ of $\OO_K$ is 
comparable in size to the diameter of the
region in $\R^n$ occurring in Theorem~\ref{firstres} where we wish to count elements $\beta\in\OO_K$.
The following is a restatement of Theorem~\ref{secondres}:

\begin{theorem}\label{secondres2}
  Let $K$ be a number field of degree $n$. With notation as introduced above, we 
  have
  $$|v_{n-1}|\ll |\Disc(K)|^{1/n}.$$
\end{theorem}

\begin{proof}
Consider the $(n-2)
\times (n-2)$ symmetric matrix $A:=(a_{ij})$, where $a_{ij}$ is 
the coefficient of $v_{n-1}$ in the expansion of $v_iv_j$ with respect to the basis $1, v_1,\dots,v_{n-1}$. We claim that $A$ 
is nondegenerate. For if not, then there is some nonzero element $r$
in the span of $v_1,\dots,v_{n-2}$ such that multiplication by $r$
leaves invariant the $\Q$-vector space $L$ spanned by $1,v_1, \dots,v_{n-2}$. 
This means that multiplication by the field $\Q(r)$ leaves $L$ invariant; i.e., $L$ is a $\Q(r)$-vector subspace of~$K$.  But this is a contradiction, since the $\Q$-dimension of $\Q(r)$ is at least 2, while $L$ has codimension 1 in $K$ as a $\Q$-vector space.  Thus $A$ is nondegenerate.

It follows that there is some permutation $\pi\in
S_{n-2}$ such that $a_{i,\pi(i)}\neq 0$ for all $1\leq i\leq
n-2$. 
As, for each $i$, both $v_{n - 1}$ and $v_i v_{\pi (i)}$ complete $\{1, v_1, \dots, v_{n - 2}\}$ to a full rank lattice, we conclude
by construction that
$|v_{n-1}|\ll |v_iv_{\pi(i)}|$. Multiplying over all $i$ yields
$|v_{n-1}|^{n-2}\ll \prod_{i=1}^{n-2} |v_i|^2$, and then multiplying both
sides by $|v_{n-1}|^2$ gives
$$|v_{n-1}|^n\ll \prod_{i=1}^{n-1} |v_i|^2 \ll |\Disc(K)|\,,$$ as desired.
\end{proof}

\section{Bounding the 2-torsion in class groups of number fields of arbitrary degree}

We now show how Theorems~\ref{firstres} and \ref{secondres} imply
Theorem~\ref{main}.  Indeed, let $B$ denote the box in
Theorem~\ref{firstres}, viewed inside $\R^n$, containing the set of $\beta$ we wish to
count, i.e., containing all elements of $\OO_K\subset \R^n$ whose archimedean
embeddings all have size $\ll |\Disc(K)|^{1/n}$.
Theorem~\ref{secondres} shows that, for a positive constant $C = C(n)$
depending only on the degree $n$ of $K$, the dilated box $CB$ covers the
fundamental domain $F$ for $\OO_K$ in $\R^n$ defined by its Minkowski basis.
Thus, the number of possible $\beta\in \OO_K$ lying in $B$ is at most 
$O(\Vol(B)/\Vol(F)) = O(|\Disc(K)|^{1/2})$, which is already better
than the trivial bound coming from Brauer--Siegel. To prove
Theorem~\ref{main}, however, we require a further power saving on this
bound, which we will obtain by giving an upper estimate for the number of just those $\beta$ 
whose norm is a square.

First, we handle the case where $K$ contains a subfield $F$ of index $2$.  In that case, by genus theory, the 2-torsion in the class group of $K$
 is bounded by the 2-torsion in the class group of~$F$ multiplied by $O(2^t)$, where $t$ is the number of ramified primes in $K/F$. Since $2^t$ is subpolynomial in $|\Disc(K)|$, and $|\Disc(F)|\leq |\Disc(K)|^{1/2}$, we 
 obtain a bound of $O_{\epsilon}(|\Disc(K)|^{1/4+\epsilon})$ in this case, by Brauer--Siegel applied to $F$ together with genus theory for $K/F$.
 
Hence we will assume from now on that $K$ does not contain any subfield of index $2$. For $\beta\in B\cap \OO_K$, we consider the line formed by $\beta+\Z$, and the polynomial $f_{\beta}(m)=\textrm{Norm}_{K/\Q}(\beta-m)$.

\begin{lemma} 
If $K$ does not contain an index $2$ subfield, then the number of $\beta\in B\cap \OO_K$ such that $f_{\beta}(x)$ is a square in $\Q[x]$ is at most $O(|\Disc(K)|^{1/4})$.
\end{lemma}

\begin{proof}
The roots of $f_{\beta}$ are the Galois conjugates of $\beta$, and so $f_{\beta}$ is squarefree as a polynomial unless $\beta$ is contained in some proper subfield $F$ of $K$, in which case $f_\beta$ is a square if and only if $[K:F]>2$ is even. 
For any such subfield $F$, we have $|\Disc(F)|\leq |\Disc(K)|^{1/[K:F]}$, so we may again apply Theorem~\ref{secondres} to count the number of $\beta\in\OO_F\cap B$.  Indeed, $\OO_F$ is an $[F:\Q]$-dimensional lattice with covolume $\asymp |\Disc(F)|^{1/2}$ and largest successive minimum of length $\ll |\Disc(F)|^{1/[F:\Q]}\leq |\Disc(K)|^{1/n}$, and we wish to count the number of $\beta$ in the lattice $\OO_F$  in a region of diameter $\asymp |\Disc(K)|^{1/n}$.  The number of such $\beta$ is thus
$$\ll \frac{(|\Disc(K)|^{1/n})^{[F:\Q]}}{\;|\Disc(F)|^{1/2}} \leq |\Disc(K)|^{1/[K:F]} \leq |\Disc(K)|^{1/4}.$$
The number of possible such subfields $F$ of $K$ is at most a bounded constant depending only on $n$; summing up over all such subfields $F$ of $K$  yields the desired result. \end{proof}

We now use a result of Bombieri--Pila~\cite{BP}. If $f_{\beta}$ is not a square, then the number of integers $m\ll |\Disc(K)|^{1/n}$ with $f_{\beta}(m)$ equal to a square is the same as the number
of integral points on the (irreducible) curve $C_\beta:y^2=g_{\beta}(x)$ (where $g_\beta$ denotes the squarefree part of $f_\beta$) satisfying $|x|\ll |\Disc(K)|^{1/n}$ and $|y|\ll |\Disc(K)|^{1/2}$. By \cite[Theorem~5]{BP} (taking there $N=O(|\Disc(K)|^{1/2})$, the number of integral points on $C_\beta$ lying in this region of $(x,y)$ is at most
$O_{\epsilon}(|\Disc(K)|^{\frac1{2n}+\epsilon}$). 

Summing up over all $O(|\Disc(K)|^{\frac12-\frac1{n}}$)
representatives $\beta$ in $(\OO_K\cap B)/\Z$ then gives the desired bound of 
$O_{\epsilon}(|\Disc(K)|^{\frac12-\frac1{2n}+\epsilon}$) for the number of 2-torsion elements in the class group of~$K$. We have proven Theorem~\ref{main}.

\begin{remark}
{\em 
One may also use a sieve,  such as the Selberg sieve or the square sieve, to obtain a power saving on $O(|\Disc(K)|^{1/2})$, but such sieve methods do not seem to readily yield power savings that are as strong.
}
\end{remark}

In the case of degree $n=3$, we can in fact improve our bounds for the count of such $\beta$ that have square norm by relating this count to the number of integral points on elliptic curves. 
The same improved bound for the 2-torsion in the class groups of quartic fields then follows formally by the work in \cite{T}.
These improvements to Theorem~\ref{main} in the cubic and quartic cases, along with the proof of Theorem~\ref{main2}, are carried out in the next section.

\section{Specializing to the cubic and quartic cases, and integral points on elliptic curves}

In this section, we further improve our bounds on the 2-torsion in the class groups of cubic fields, by relating the problem, via the work of the previous section, to counting integral points on elliptic curves.  We then appeal to the bounds of Helfgott and Venkatesh~\cite{HV} on integral points and of Brumer and Kramer~\cite{BK} on ranks of elliptic curves.  As a consequence, we also then obtain nontrivial improvements to the best known bounds on 2-Selmer groups, ranks, and the number of integral points on elliptic curves, proving Theorem~\ref{main2}. Appealing to \cite{T} then yields the same bounds also on the 2-torsion in class groups of quartic fields.

\subsection{The cubic case of Theorem \ref{main}}

Let $K$ be a cubic field. Expanding the polynomial $f_{\beta}(m)$ from the previous section, it follows that $h_2(K)$ is bounded above by the number of integral points $(x,y)$ on 
$O(|\Disc(K)|^{1/6})$ elliptic curves $E_{\pm,A,B,C}$ of the form $\pm y^2= x^3+Ax^2+Bx+C$, where $|A|\leq |\Disc(K)|^{1/3}$, $B\leq |\Disc(K)|^{2/3}$, and $C\leq |\Disc(K)|$, that satisfy $|x|\ll |\Disc(K)|^{1/3}$ and $|y|\ll |\Disc(K)|^{1/2}$. 
It therefore suffices to bound the number of such integral points on each of these curves uniformly.

To accomplish this, we recall the following immediate consequence of a result
of Helfgott and Venkatesh \cite[Corollary~3.11]{HV},
who delicately apply sphere packing bounds 
to give a bound on the total number of integral points on an elliptic curve over $\Q$ in terms of its discriminant and rank:

\begin{theorem}\label{hvthm}
Let $E$ be an elliptic curve over $\Q$ in Weierstrass form with integral coefficients. 
Then the number of integer points on $E$ is
\begin{equation}\label{eqn:HV}
O_{\epsilon}\Big( \,|\Disc(E)|^\epsilon\, e^{(\beta + \epsilon) \rk(E)} \Big),
\end{equation}
where
\begin{equation}\label{def:beta}
\alpha = \frac{\sqrt{3}}{2}, \ \ \ \ \ \ 
\beta
= \frac{1 + \alpha}{2 \alpha} 
\log \frac{1 + \alpha}{2 \alpha}  - \frac{1 - \alpha}{2 \alpha} 
\log \frac{1 - \alpha}{2 \alpha} = .2782\dots
\end{equation}
\end{theorem}

Helfgott and Venkatesh actually prove a more general theorem~\cite[Theorem~3.8]{HV}, which allows for the choice of a parameter $t \in [0, 1]$ (to interpolate between sphere packing bounds and Bombieri--Pila-style bounds) and also allows as input an upper bound on the canonical height of the integer points to be counted.  
We have checked that the choice $t = 0$ (i.e., a pure application of sphere packing) in their theorem optimizes our ensuing bounds; in addition, imposing our restrictions from the previous section on the canonical height of the integral points we wish to count on $E_{\pm,A,B,C}$ (when $t=0$) 
does not improve the exponent of $|\Disc(K)|$ in our ensuing bounds.  Therefore, 
Theorem~\ref{hvthm}
represents the special case of \cite[Theorem~3.8]{HV} that yields the optimal bounds in our application. 

We also apply the following result of Brumer and Kramer (see~\cite[Proposition 7.1 and subsequent remark]{BK}) relating Selmer groups (and thus ranks) to class groups.
\begin{theorem} \label{thm:BK}
Suppose that $E$ is an elliptic curve defined by a Weierstrass equation $y^2 = f(x)$ over $\Q$.
Let $K$ be the \'etale cubic $\Q$-algebra $\Q[x]/(f(x))$.
Then
\begin{equation}\label{eqn:BK}
\rank(E)\leq \rk_2(\Sel_2(E)) \leq \rk_2(\Cl(K)[2]) + 2 \omega(|\Disc(E)|) + 2,
\end{equation}
 where as usual $\omega(n)$ is the number of distinct prime divisors of $n$.
\end{theorem}

Thus, we have shown that the 2-torsion in the class group of a cubic field $K$ is bounded in terms of the number of integral points on a collection of elliptic curves $E_{\pm,A,B,C}$, each having a 2-torsion point defined over $K$; meanwhile, Theorems~\ref{hvthm} and~\ref{thm:BK} together show that the number of integral points on an elliptic curve with a 2-torsion point defined over $K$ is in fact bounded in terms of the 2-torsion in the class group of $K$!  This results in a feedback loop, where each improved bound on the 2-torsion in the class groups of cubic fields results in an improved bound on the ranks and number of integral points on elliptic curves over $\Q$, and vice versa. 

To solve for the point of convergence, we substitute the bound (\ref{eqn:BK}) of Theorem~\ref{thm:BK}
into Equation~\eqref{eqn:HV} of Theorem \ref{hvthm}, and sum over our $|\Disc(K)|^{1/6}$ elliptic curves $E_{\pm,A,B,C}$ each of which has absolute discriminant $\ll |\Disc(K)|^2$ by the bounds on $A,$ $B,$ and $C$.  We thereby obtain 
$$h_2(K)\ll_\epsilon |\Disc(K)|^{1/6+\epsilon}e^{(\beta + \epsilon) \log_2h_2(K)}$$
with $h_2(K)$ now 
occuring on both sides of the inequality. 
We conclude that
\begin{equation}\label{cl2bound}
h_2(K)\ll_{\epsilon}|\Disc(K)|^{{\textstyle{\frac{1}{6\left(1-\frac{\beta}{\log 2}\right)}}}+\epsilon}= O_\epsilon(|\Disc(K)|^{.2784\ldots+\epsilon}),
\end{equation}
proving Theorem~\ref{main} for cubic fields.

\subsection{Proof of Theorem \ref{main2} on  bounds for $2$-Selmer groups, ranks, and integral points on elliptic curves}
Parts~(a) and (b) of Theorem~\ref{main2}, on 2-Selmer groups and ranks of elliptic curves, are now immediate from the inequality \eqref{eqn:BK} of Brumer and Kramer and the bound (\ref{cl2bound}) on the size of the 2-torsion in class groups of cubic fields. 

Meanwhile, Part (c) of Theorem~\ref{main2}, on the total number of integral points on an elliptic curve $E$ over $\Q$, follows by substituting~the result for Part~(b) into~Equation~\eqref{eqn:HV} of Theorem~\ref{hvthm}. We remark that the previously best known bound for the total number of integral points on $E$, of $O_{\epsilon}(|\Disc(E)|^{.2007\ldots + \epsilon})$,
is~Corollary~3.12 of \cite{HV}, proved in the same way but using the trivial bound
$h_2(K) \leq h(K)$.

\subsection{The quartic case of Theorem \ref{main}}

To handle the case of quartic fields, we recall \cite[Lemma 6.1]{T}:

\begin{lemma}\label{quarcub}
Let $K$ be a quartic field of discriminant $\Disc(K)$ with a cubic resolvent field $F$. Then the numbers $h_2(K)$ and $h_2(F)$  are the same up to a multiplicative constant of order $|\Disc(K)|^{o(1)}$.
\end{lemma}

\begin{remark}
{\em 
The above lemma in fact holds for any quartic field---even if its cubic resolvent is not a field but merely an \'etale cubic algebra---with the identical proof.
}
\end{remark}

Since the discriminant of $F$ is at most the discriminant of $K$, we conclude immediately from Lemma~\ref{quarcub} that 
$$h_2(K) \ll_{\epsilon}|\Disc(K)|^{{\textstyle{\frac{1}{6\left(1-\frac{\beta(0)}{\log 2}\right)}}}+\epsilon}= O_\epsilon(|\Disc(K)|^{.2784\ldots+\epsilon}),$$
giving the same improvement of the power saving in Theorem~\ref{main} for quartic fields as well.

\section{The number of $A_4$-quartic fields of  bounded discriminant}\label{a4sec}

We can also use our bounds on the 2-torsion in class groups of cubic fields to improve on the best known bounds, due to Baily~\cite{Baily} and Wong~\cite{Wong}, on the number of quartic fields (up to isomorphism) of bounded discriminant having associated Galois group the alternating group $A_4$.  

Let $\mathcal{F}_3$ be any set of isomorphism classes of cubic fields, and let $\mathcal{F}_4$ be the subset of isomorphism classes of
$A_4$- and/or $S_4$-quartic fields whose cubic resolvents are in $\mathcal{F}_3$.
Write $N(\mathcal{F}_4, X)$ for the number of quartic fields in $\mathcal{F}_4$ with absolute discriminant
less than $X$.

In~\cite[Lemma 10 and Theorem 2]{Baily}, Baily proved the estimate
\begin{equation}\label{beq}
N(\mathcal{F}_4, X) \ll \sum_{{\scriptstyle{K \in \mathcal{F}_3}\atop\scriptstyle{|\Disc(K)|<X}}} h_2(K) \log^2(|\Disc(K)|) \left( \frac{X}{|\Disc(K)|} \right)^{1/2}.
\end{equation}
Baily only considered the cases where $\mathcal{F}_3$ is the set of all noncyclic cubic fields (in his Theorem 2) and
all cyclic fields (his Theorem 4), but his arguments evidently give the more general result \eqref{beq} without any change.

For a transitive subgroup $G\subset S_n$, let $N_n(G,X)$ denote the number of isomorphism classes of number fields of degree $n$ having Galois closure with Galois group $G$ and absolute discriminant less than $X$. Incorporating the bound
\begin{equation}
h_2(K) \leq h(K) \ll |\Disc(K)|^{1/2} \log^2(|\Disc(K)|)
\end{equation}
and the asymptotics $N_3(C_3, X) \sim C_1 X^{1/2}$ and $N_3(S_3, X) \sim C_2 X$, due to Cohn \cite{Cohn}
and Davenport and Heilbronn \cite{DH}, respectively, Baily concluded that
\begin{equation}
N_4(A_4, X) \ll X \log^4 X, \ \ \ N_4(S_4,X) \ll X^{3/2} \log^4 X.
\end{equation}

These estimates were improved by Wong~\cite{Wong0, Wong}, who showed {\it on average} that
$h_2(K)=O_\epsilon(|\Disc(K)|^{1/3+\epsilon})$, yielding
the improved bound $N_4(A_4,X)=O_\epsilon(X^{5/6+\epsilon})$.  

Our Theorem~\ref{main} gives the improved bound $h_2(K)=O_\epsilon(|\Disc(K)|^{.2784\ldots+\epsilon})$ for individual cubic fields $K$ (regardless of whether they are cyclic or noncyclic).
Substituting this bound into~\eqref{beq} and carrying out the partial summation as in \cite{Baily}, we immediately obtain that
\begin{equation}
N_4(A_4,X)=O_\epsilon(X^{.7784\ldots+\epsilon}),
\end{equation}
proving Theorem~\ref{main3}.

\medskip
\begin{remark}\label{quadresolve}
{\em In fact, for any fixed quadratic field $F$, let $N_n(G,F,X)$ denote the number of isomorphism classes of number of fields of degree $n$ having Galois closure with Galois group $G$, quadratic resolvent field isomorphic to $F$, and absolute discriminant less than $X$.  
(For example, if $G=S_4$, then $N_4(S_4,F,X)$ counts the number of quartic fields $K$ of absolute discriminant less than $X$ such that $F$ is the unique quadratic subfield of the Galois closure of $K$.) We obtain the bound
\begin{equation}\label{eqn:S4K}
N_4(S_4,F,X)=O_{\epsilon, F}(X^{.7784\ldots+\epsilon}).
\end{equation}
To prove this, in place of Cohn's result we use the asymptotic
\begin{equation}\label{eqn:BSCM}
N_3(S_3,F, X) \sim C_F X^{1/2},
\end{equation}
proved independently by the first author and Shnidman \cite[Theorem 7]{BShnid} 
and Cohen and Morra \cite[Theorem 1.1(2), Corollary 7.6]{CM}, and the result follows immediately.
We remark that $C_F$ is $O\big(|\Disc(F)|^{-1/2} \cdot L(1, \chi(F'))\big)$, where $F'$
is the `mirror field' of $F$, so that the implied constant in~\eqref{eqn:S4K} is decreasing 
in $|\Disc(F)|$. However, the rate of convergence in \eqref{eqn:BSCM} depends on $F$, so that 
we do not obtain uniformity in the quadratic field $F$ (at least not without further effort).}
\end{remark}

\section{Bounding 2-torsion in class groups of function fields}\label{ffsec}

There is a well-known analogy between number fields and function fields of curves, which we briefly review. Let $k$ be a finite field. For the purposes of this paper, when we say $C$ is a curve over $k$, we mean that $C$ is a smooth, projective, geometrically irreducible curve over $k$.  Let $g$ be the genus of~$C$. Recall that the gonality of $C$
is defined to be the minimal degree of a non-constant map $\phi:C\to \P^1$ defined over $k$. Curves of genus $g$ and gonality~$n$ are analogous to number fields of degree $n$ over $\Q$ with discriminant of size roughly $|k|^{2g}$. In particular, it is noteworthy that when we speak of all curves of genus $g$, it is analogous to speaking of all number fields of discriminant roughly of size $|k|^{2g}$, \emph{without fixing their degree}! 
We now prove our main theorem:

\begin{theorem}\label{genus}
Let $C$ be a curve of genus $g$ over a finite field $k$. Then $$|\pic^0(C)(k)[2]|\leq \frac{|k|^{g+1}-1}{|k|-1}.$$
Moreover, if $C$ admits a degree $n$ map to $\P^1$ over $k$, then we have $$|\pic^0(C)(k)[2]|\ll_n |k|^{(1-\frac{1}{n})g}.$$
\end{theorem}

\begin{remark} 
{\em 
It may seem that the first part of our theorem follows immediately from the class number formula $$|\pic^0(C)(k)| = |k|^g\log|k|\res_{s=1}\zeta_C(s),$$ at least up to a sub-exponential factor in $g$. However, the point of this theorem is that we are not restricting the gonality of $C$, which means that $\res_{s=1}\zeta_C(s)$ can in fact grow (or shrink) exponentially in $g$. This is analogous to the fact that the usual Brauer--Siegel theorem over number fields becomes false if one drops the requirement that the discriminant must grow super-exponentially in the degree. In fact, the trivial bound is $(1+|k|^{0.5})^{2g}$. Thus, the first part of our theorem does indeed provide a nontrivial power saving.
}
\end{remark}

\begin{remark}
{\em 
For function fields, there is another way to get a ``trivial bound'', namely by noting that $\pic^0(C)(k)[2]\subset \pic^0(C)(\overline{k})[2]$ and the latter group has size $2^{2g}$, or $2^g$ if
the characteristic is~$2$. This (usually) does better than the other trivial bound, and in fact our result only beats this bound in the case where $k=\F_3$. In this case, 
to our knowledge this is the first such improvement, where the problem is usually as difficult as in the number field case.
}
\end{remark}

\begin{remark} {\em The result we obtain in the function field case is highly suggestive that in the number field case, it should be provable by the analogous means that the 2-torsion in the class group is bounded by the square root of the discriminant, regardless of the degree. However, as far as we can see, in trying to carry out the proof we get derailed by the complexities of high-dimensional geometry. In other words, we have to replace the straightforward use of Riemann--Roch with the difficulties of Minkowski bases in high dimensions. }
\end{remark}

\begin{proof}
We fix a degree $g$ divisor $E$ of $C$. Now, suppose that $D$ is a degree $0$ divisor of $C$ , such that $2D$ is principal. By the Riemann-Roch theorem, any divisor of degree $g$ or higher has a section, and thus, $h^0(C,E-D)>0$. Hence, at the cost of replacing $D$ by a linearly equivalent divisor, we may assume that $E-D$ is a positive divisor $D_+$, and thus we may write $D=E-D_+$. Now as 2D is principal, it follows that $h^0(C,2E-2D_+)=1$ or, in other words, $H^0(C,2E)$ has an
element~$f$, unique up to scaling,  such that $\div(f) = 2D_+-2E$. Let $$H^0(C,2E)^{\rm sq}:=\{f\in H^0(C,2E): f\neq 0 \mbox{ and } 2\mid\div(f)\}.$$ It follows that there is a surjection from $\P H^0(C,2E)^{\rm sq}$ onto 
$\pic^0(C)(k)[2]$. The Riemann--Roch theorem, combined with the fact that the first cohomology of any divisor vanishes if the degree of that divisor is larger than $2g-2$, now implies that $|H^0(C,2E)| = |k|^{g+1}$. This implies the first claim in our theorem.

For the second claim, we proceed just as in the proof of Theorem \ref{main}, but using instead the function-field analogue of the theorem of Bombieri and Pila~\cite[Theorem~5]{BP} established by Sedunova~\cite[Theorem~1]{Sedunova}.
\end{proof}

\subsection*{Acknowledgements}

This paper is the result of our collaboration at a SQuaRE (Structured Quartet Research Ensemble) held at the American Institute of Mathematics. We would all like to thank AIM for their financial support and for providing an excellent environment in which to work.

MB was partially supported by a Simons Investigator Grant and NSF Grant DMS-1001828.
TT  was partially supported by the JSPS and KAKENHI Grants JP24654005 and JP25707002.
FT  was partially supported by NSF Grant DMS-1201330 and by the National Security Agency under a Young Investigator Grant.

Finally, we would like to thank Bob Hough for helpful conversations.

\end{document}